\documentclass{article}
\usepackage{amsthm}
\usepackage{amsfonts}
\usepackage{amssymb}
\usepackage{mathrsfs}
\usepackage{amsmath}
\usepackage{graphicx} 
\usepackage{hyperref}
\usepackage{xcolor}
\hypersetup{
	colorlinks,
	linkcolor={black},
	citecolor={black},
	urlcolor={black}
}
\usepackage{cleveref}
\newtheorem{theorem}{Theorem}[section]
\newtheorem{lemma}[theorem]{Lemma}

\newtheorem{proposition}[theorem]{Proposition}
\newtheorem{corollary}[theorem]{Corollary}
\newtheorem{definition}[theorem]{Definition}

\newtheorem{remark}[theorem]{Remark}

\newtheorem{example}[theorem]{Example}
\newtheorem{question}[theorem]{Question}

\newcommand{\N}{\mathbb{N}}

\newcommand{\Z}{\mathbb{Z}}

\newcommand{\Dp}{\operatorname{DP}}

\begin{document}

\title{On a Rice theorem for dynamical properties of SFTs on groups}
\author{Nicanor Carrasco-Vargas\footnote{Mathematics faculty, Jagiellonian University, Krak\'ow, Poland; \href{mailto:nicanor.vargas@uj.edu.pl}{nicanor.vargas@uj.edu.pl}.}}
\date{}
\maketitle
\begin{abstract}
	Let $G$ be a group with undecidable domino problem, such as $\mathbb{Z}^2$. We prove that all nontrivial dynamical properties for sofic $G$-subshifts are undecidable, that this is not true for $G$-SFTs, and an undecidability result for dynamical properties of $G$-SFTs similar to the Adian-Rabin theorem. Furthermore we prove a Rice-like result for dynamical invariants asserting that every computable real-valued invariant for $G$-SFTs that is monotone by disjoint unions and products is constant. 
\end{abstract}	
	\section{Introduction}\label{sec:Introduction}
	
	There has been a recent interest in extending the study of SFTs on $\Z$ to SFTs on other groups. An important obstruction that has arisen for this project is of an algorithmic nature. Indeed, many dynamical questions become undecidable when we move from $\Z$-SFTs to $\Z^2$-SFTs. Lind \cite{lind_multidimensional_2004} called this the  
	``swamp of undecidability''. 
	One way to better understand this swamp is to ask whether we
	have a result analogous to Rice's theorem, which states that \textit{every} nontrivial semantic property of computer programs is algorithmically undecidable
	\cite{rice_classes_1953}. Nontrivial means that some element satisfies the property and some element does not. This result has been paradigmatic in the
	sense that \emph{Rice-like theorems} have been discovered in a variety
	of mathematical contexts \cite{kari_rice_1994,delacourt_rice_2011,menibus_characterization_2018,hirst_ricestyle_2009,lafitte_computability_2008,Gamard2021RiceLikeTF,delvenne_quasiperiodic_2004,rice_classes_1953,adian_decision_2000,rabin_recursive_1958,guillon_revisiting_2010,adyan_algorithmic_1955}. In this work we consider the following question:
	
	\begin{question}
		\label{que:main-question} Is there a Rice theorem for dynamical properties of SFTs on $\Z^2$? What about SFTs on other groups?
	\end{question}
	That is, the goal of this note is converting the metaphor \textit{swamp of undecidability} from \cite{lind_multidimensional_2004} to precise mathematical statements. 
	
	\subsubsection*{Known results}
	These results are stated in terms of sets of tilings of $\mathbb{Z}^{2}$. We recall that a \textbf{tile }(or
	Wang tile) is a unit square with colored edges, a \textbf{tileset}
	$\tau$ is a finite set of tiles, a \textbf{tiling} is a function $\Z^{2}\to\tau$ satisfying the rule that two tiles
	that share an edge must have the same color at that edge, and a \textbf{set of tilings} is the collection of all tilings associated to a tileset. Sets of tilings of $\Z^{2}$ and $\Z^{2}$-SFTs are related by the fact that every set of tilings of $\Z^{2}$ is a $\Z^{2}$-SFT and every $\Z^{2}$-SFT
	is topologically conjugate to a set of tilings \cite{aubrun_domino_2018}.
	\begin{enumerate}
		\item In \cite{cervelle_tilings_2004} the authors prove the
		undecidability of set equality and set inclusion for sets
		of tilings of $\Z^{2}$, and ask for a Rice theorem for tilings.
		\item In \cite{lafitte_computability_2008} the authors consider
		properties of sets of tilings of $\Z^{2}$ that are preserved by ``zoom''.
		Informally two tilesets are equivalent in this sense when the
		tiles from each tileset produce unique macro-tiles that satisfy the
		same matching rules from the other tileset. The authors prove that
		all nontrivial properties of sets of tilings  preserved by ``zoom''
		are undecidable.
		\item In \cite{delvenne_quasiperiodic_2004} the authors prove
		the undecidability of every property of sets of tilings of $\Z^{2}$
		that is stable by topological conjugacy, by direct
		products among nonempty systems, and is not satisfied by the
		empty set. 
	\end{enumerate}
	
	\subsubsection*{Results}
	In this work we consider dynamical properties of SFTs, that is, properties preserved by topological conjugacy. We investigate the (un)decidability of these properties for SFTs on finitely generated groups. A first observation is that a Rice theorem is not possible in this setting: the property of having some fixed point is nontrivial and  decidable from presentations (\Cref{prop:decidable-fixed-points}). Thus we need to add some hypotheses to the properties considered in order to prove a Rice-like theorem.
	
	Our undecidability results are related to the domino problem, the algorithmic problem  of determining whether an SFT is empty. Berger \cite{berger_undecidability_1966} proved that the domino
	problem is undecidable for $\Z^{d}$, $d\geq2$, and this has been
	recently extended to a large class of finitely generated groups \cite{aubrun_domino_2019,aubrun_tiling_2013,ballier_domino_2018,bartholdi_domino_2023,bitar_contributions_2023,cohen_large_2017,jeandel_translationlike_2015,jeandel_undecidability_2020,margenstern_domino_2008,aubrun_strongly_2023}.
	In \cite{ballier_domino_2018} the authors conjecture that all non virtually free groups have
	undecidable domino problem. 
	
	Our main result states that if $G$ has undecidable domino problem, then all dynamical properties of $G$-SFTs that satisfy a mild technical condition are undecidable (\Cref{thm:adian-rabin-for-sfts}). Our result covers several dynamical properties of common interest, such as transitivity, minimality, and others (see \Cref{sec:examples}), and can be applied to all nontrivial dynamical properties that
	are preserved by topological factor maps or topological extensions (\Cref{cor:rice-for-monotone-properties}). We mention that \Cref{thm:adian-rabin-for-sfts} exhibits an analogy with the
	Adian-Rabin theorem for group properties \cite{adyan_algorithmic_1955,rabin_recursive_1958}, the ``Rice-like
	theorem in group theory''. A number of analogies have been observed between group theory and symbolic dynamics \cite{jeandel_enumeration_2017,jeandel_characterization_2019,jeandel_computability_2016}. 
	
	We also consider dynamical invariants of SFTs taking values in partially ordered
	sets. We show that if $G$ has undecidable domino problem then every abstract real-valued dynamical invariant for $G$-SFTs that is nondrecreasing by disjoint unions and direct products must be constant  (\Cref{invariants-monotone-by-products-and-unions-are-constants}). This result covers topological entropy for amenable groups and other related invariants. We mention that for some amenable groups it is known a much stronger fact: the existence of SFTs whose entropy is a non-computable real number \cite{hochman_characterization_2010,barbieri_entropies_2021,bartholdi_shifts_2024}. The conclusion of \Cref{invariants-monotone-by-products-and-unions-are-constants} is much weaker, but it is a much more general result. 
	
	We also consider the larger class of sofic subshifts. We prove that if the domino problem for $G$ is undecidable, then a Rice theorem for dynamical properties of sofic subshift holds. That is, all nontrivial
	dynamical properties are undecidable (\Cref{thm:Rice-theorem-sofic-subshifts}).

	All our undecidability results are proved through a many-one reduction to the domino problem for $G$. In informal words our results show that most dynamical properties are harder than the domino problem. As in the original Rice's theorem our proofs are very short. They are  based on the computability of direct products and disjoint unions at the level of presentations. In contrast, although tiles can be defined on Cayley graphs of finitely generated groups, the proofs of the results in \cite{lafitte_computability_2008,cervelle_tilings_2004}
	are specific to $\Z^{2}$, and it is unclear whether they could
	be generalized to a group that is not residually finite. 
	\section{Preliminaries} We start reviewing shift spaces, see also \cite{ceccherini-silberstein_cellular_2010}. Let $G$ be a finitely generated group and $A$ be a finite alphabet. We endow $A^G=\{x\colon G\to A\}$ with the prodiscrete topology and the continuous shift action $G\curvearrowright A^G$ defined by $(gx)(h)=x(g^{-1}h)$. A closed and shift-invariant subset of $A^G$ is called a \textbf{subshift}. A \textbf{morphism} of subshifts $X\subset A^G$ and $Y\subset B^G$ is a map $\phi\colon X\to Y$ that is continuous and commutes with the shift action. We call it \textbf{conjugacy} when it is bijective, \textbf{embedding} when it is injective, and \textbf{factor} when it is surjective. In the last case we also say that $Y$ is a \textbf{factor} of $X$ and that $X$ \textbf{extends} $Y$. By the Curtis-Hedlund-Lyndon theorem $\phi$ is a morphism if and only if there is a finite set $D\subset G$ and a \textbf{local function} $\mu\colon A^{D}\to B$ such that $\phi(x)(g)=\mu((g^{-1}x)|_{D})$ for all $g\in G$. A property of subshifts is called a \textbf{dynamical property} when it is invariant by conjugacy.

	We now review subshifts of finite type (SFTs) and their {presentations}, see also  \cite{aubrun_notion_2017,aubrun_domino_2018}. Let $S$ be a finite and symmetric generating set for $G$. Given a word $w\in S^{\ast}$ we denote by $\underline{w}$ the corresponding group element. 
	\begin{definition}
		A \textbf{pattern presentation }is a function $p\colon W\to A$, where
		$W\subset S^{\ast}$ is a finite set of words and $A\subset\N$ is
		a finite alphabet. We say that $p$ \textbf{appears }in $x\in A^{G}$
		at $g\in G$ when $x(g\underline{w})=p(w)$ for every $w\in W$. An \textbf{SFT presentation }is a tuple $(A,\mathcal{F})$ of a finite alphabet 
		$A\subset \N$ and a finite set $\mathcal{F}$ of pattern
		presentations associated to $S$. It determines the subshift $X_{(A,\mathcal{F})}$
		of all configurations in $A^{G}$ where no pattern presentation
		from $\mathcal{F}$ appears. We call $X_{(A,\mathcal{F})}$ a \textbf{subshift of finite type} (SFT). The \textbf{domino problem }$\Dp(G)$ is the set of all presentations
		$(A,\mathcal{F})$ such that $X_{(A,\mathcal{F})}$ is
		empty. %
	\end{definition}
	A survey on the domino problem can be found in  \cite{aubrun_domino_2018}. Observe that the empty set is an SFT with our definitions. Some
	authors consider the empty set as a subshift \cite{cohen_large_2017},
	while others exclude it by definition \cite{aliakbar_set_2018}. Here
	we follow the first convention. This is a natural choice in our setting: an algorithm
	able to detect a particular property from $G$-SFT presentations should give the same output when given presentations of the empty subshift. Furthermore we can not exclude these presentations without assuming that the domino
	problem for $G$ is decidable. A consequence of this convention is
	that a dynamical property must assign yes/no value to the empty subshift.
	
	When we consider the decidability of a dynamical property
	$\mathscr{P}$ from presentations of $G$-SFTs we formally refer to the set $\{(A,\mathcal{F})\mid X_{(A,\mathcal{F})}\text{ satisfies }\mathscr{P}\}$. Our undecidability results will follow from many-one reductions of the form $\mathscr{P}\geq_m\Dp(G)$, but this notion is not required to understand the proofs.  
	\section{Undecidability of dynamical properties of SFTs }
	
	In this section we study the undecidability of dynamical properties
	of SFTs. The following example shows that a Rice theorem fails in this setting. 
	\begin{proposition}
		\label{prop:decidable-fixed-points} Let $G$ be a finitely generated
		group. Then the property of $G$-SFTs of containing a fixed point is decidable
		from presentations. 
	\end{proposition}
	
	\begin{proof}
		Recall that a fixed point of an SFT is a configuration $x$ such that $gx=x$ for all $g\in G$. Let $(A,\mathcal{F})$ be a $G$-SFT presentation, and for each $a\in A$ 
		denote by $x_{a}\colon G\to A$ the configuration with constant value
		$a$. Note that a fixed point in $X_{(A,\mathcal{F})}$
		is equal to $x_{a}$ for some $a\in A$. Clearly a pattern
		presentation $p\colon W\to A$ appears in $x_{a}$ if and only 
		if $p$ has constant value $a$, and this is
		a decidable property of $p$. Thus the following algorithm proves the statement: on input $(A,\mathcal{F})$ check whether for some $a\in A$ the set $\mathcal{F}$ fails
		to contain a pattern presentation with constant value $a$.
	\end{proof}
	In order to prove our undecidability results we need to verify the computability of direct products and disjoint unions at the level of presentations. 
	\begin{lemma}
		\label{lem:products}Let $G$ be a finitely generated group. There is an effective procedure which, given
		two presentations $(A,\mathcal{F})$ and $(B,\mathcal{C})$ of $G$-SFTs,
		outputs a presentation of a $G$-SFT that is topologically conjugate
		to the direct product $X_{(A,\mathcal{F})}\times X_{(B,\mathcal{C})}$. 
	\end{lemma}
	
	\begin{proof}
		Let $\alpha\colon\N^{2}\to\N$ be a computable bijection, and let
		$\pi_{1},\pi_{2}\colon\N\to\N$ be the computable functions defined
		by $\pi_{i}(\alpha(n_{1},n_{2}))=n_{i}$, $i=1,2$. On input $(A,\mathcal{F})$
		and $(B,\mathcal{C})$, our algorithm starts defining the alphabet of the new SFT as $C=\{\alpha(a,b)\mid a\in A,b\in B\}\subset\N$.
		Then we compute a set $\mathcal{G}$ of pattern presentations as
		follows. For each $p\colon W\to A$ in $\mathcal{F}$ (resp. $p\colon S\to B$
		in $\mathcal{C}$), we add to $\mathcal{G}$ every function $q\colon W\to C$
		such that $\pi_{1}\circ q=p$ (resp. $\pi_{2}\circ q=p$). The output is $(C,\mathcal{G})$. It is clear that $X_{(C,\mathcal{G})}$ is conjugate
		to the direct product $X_{(A,\mathcal{F})}\times X_{(B,\mathcal{C})}$.
		Indeed, the map $\phi\colon X_{(C,\mathcal{G})} \to
		X_{(A,\mathcal{F})} \times X_{(B,\mathcal{C})}$
		given by $\phi (x)(g)=(\pi_{1} (x(g)),\pi_{2} (y(g)))$ is a topological
		conjugacy. 
	\end{proof}
	The proof of the following result is very similar, and is left to the reader.
	\begin{lemma}
		
		\label{lem:disjoint-unions}Let $G$ be a finitely generated group. There is an effective procedure which,
		given two SFT presentations $(A,\mathcal{F})$ and $(B,\mathcal{C})$,
		outputs a presentation of an SFT that is topologically conjugate to
		the disjoint union $X_{(A,\mathcal{F})}\sqcup X_{(B,\mathcal{C})}$.
	\end{lemma}
	
	\begin{remark}
		According to the definitions the direct product of two SFTs is not an SFT, but a topological
		dynamical system that is conjugate to an SFT. We will ignore this subtlety for the sake of clarity. 
	\end{remark}
	
	In the following definition we propose the term Berger property because of the analogy with Markov properties of finitely presented groups and the Adyan-Rabin undecidability theorem  \cite{lyndon_combinatorial_1977}. 
	\begin{definition}\label{def:Berger}
		A dynamical property $\mathscr{P}$ of $G$-SFTs is called a \textbf{Berger
		}property if there are two $G$-SFTs $X_{-}$ and $X_{+}$ satisfying
		the following:
		\begin{enumerate}
			\item $X_{+}$ satisfies $\mathscr{P}$.
			\item Every SFT that factors onto $X_{-}$ fails to satisfy $\mathscr P$.
			
			\item There is a morphism from $X_{+}$ to $X_{-}$. 
		\end{enumerate}
		The subshift $X_{+}$ is allowed to be empty.  
	\end{definition}
	The main result of this section is the following.
	\begin{theorem}
		\label{thm:adian-rabin-for-sfts} Let $G$ be a finitely generated group with undecidable domino problem. Then every Berger property of $G$-SFTs is undecidable. 
	\end{theorem}
	
	\begin{proof}
		Let $\mathscr{P}$ be a Berger property, and let $X_{+}$, $X_{-}$ be as in \Cref{def:Berger}. Given an SFT presentation
		$(A,\mathcal{F})$ we define $Z$ as the disjoint union of  $X_{+}$ and $X_{(A,\mathcal{F})}^ {}\times X_{-}$. We claim that $Z$ has property $\mathscr{P}$ if and only if
		$X_{(A,\mathcal{F})}^ {}$ is empty. Indeed, if $X_{(A,\mathcal{F})}$
		is empty then $Z$ is topologically conjugate to $X_{+}$. If $X_{(A,\mathcal{F})}^ {}$
		is nonempty, then $Z$ factors over $X_{-}$. This follows from two
		facts: that for $X_{(A,\mathcal{F})}$ nonempty $X_{(A,\mathcal{F})}^ {}\times X_{-}$
		factors over $X_{-}$, and that there is a topological morphism from
		$X_{+}$ to $X_{-}$. If $\mathscr{P}$ was a decidable property, then we could decide whether $X_{(A,\mathcal{F})}^ {}$
		is empty by computing a presentation for $Z$, which is possible thanks to
		\Cref{lem:products} and \Cref{lem:disjoint-unions}, and then
		checking whether $Z$ satisfies $\mathscr{P}$. This contradicts the undecidability
		of $\Dp(G)$.
	\end{proof}
	In the language of many-one reductions we proved that a Berger property $\mathscr{P}$ satisfies $\mathscr{P}\geq_m\Dp(G)$ and thus it is $\Dp(G)$-hard. Note that \Cref{thm:adian-rabin-for-sfts} applies to every nontrivial property that is preserved to topological
	factors and is satisfied by the empty subshift. Since the negation of a property preserved to
	topological factors is preserved to topological extensions, we obtain
	the following result.
	\begin{corollary}
		\label{cor:rice-for-monotone-properties} Let $G$ be a finitely generated group with undecidable domino problem. Every nontrivial dynamical property
		for $G$-SFTs which is preserved to topological factors (resp. extensions),
		and which is satisfied (resp. not satisfied) by the empty subshift,
		is undecidable. 
	\end{corollary}
	The assumption on the empty subshift is necessary: the decidable property from  \Cref{prop:decidable-fixed-points} is preserved to factors.

	\subsection{Examples}\label{sec:examples}
	Here we present some examples. We start observing that the decidability of a property could be altered if we include or exclude the empty subshift:
	
	\begin{example}\label{property-that-becomes-undecidable-after-we-allow-empty-subshift}
		Let $\mathscr P$ be the property of having some fixed point. We proved that $\mathscr{P}$ is decidable in \Cref{prop:decidable-fixed-points}. However ``$\mathscr{P}$ or empty'' is a Berger property. It suffices to take $X_{+}=\emptyset$, and $X_{-}$ as a nonempty SFT having no fixed point.   
	\end{example}
	There is a simple class of properties where this situation is prevented:
	
	\begin{remark}\label{criterion}
		Let $\mathscr P$ be a Berger property, and let $X_{+}$ and $X_{-}$ as in Definition \ref{def:Berger}. If  $X_{+}$ is nonempty, then both ``$\mathscr P$ or empty'' and ``$\mathscr P$ and nonempty'' are Berger properties. This is shown by the same pair of SFTs $X_{+}$ and $X_{-}$.
	\end{remark}
	Many commonly studied dynamical properties can be shown to satisfy this criterion, and thus they are Berger properties regardless of the value assigned to the empty subshift. Now we review a few of them. 
	\begin{example}
		A $G$-SFT is topologically transitive when it contains a configuration
		with dense orbit. Transitivity is a Berger property. Indeed, it suffices to choose $X_{+}$ as an SFT with exactly one configuration, and choose $X_{-}$ as an SFT with two fixed points. Note that topological extensions of $X_{-}$ are not transitive. 
	\end{example}
	\begin{example}
		A $G$-SFT is minimal if it has no proper nonempty subsystem. This is a Berger property by the same reasoning as in the previous example. 
	\end{example}
	
	\begin{example}
		A configuration $X\in A^G$ is called strongly aperiodic when $gx\ne x$ for every $g\in G$ different to the identity. Consider the property $\mathcal{AD}$ of having at least strongly aperiodic configuration. This property is known as the \textit{aperiodic domino problem}. The negation of $\mathcal{AD}$ is a Berger property: it suffices to take $X_{+}$ as an SFT with only one configuration, and $X_{-}=\{0,1\}^G$.  
		
		The complexity of $\mathcal{AD}$ for the group $\Z^d$ is studied in detail in  \cite{callard_aperiodic_2022}. The authors prove that this problem is $\Pi_1^0$-complete for $d=2$ and $\Sigma_1^1$-complete for $d\geq 4$. Its exact complexity is not known for $d=3$.  The argument presented here only shows that $\mathcal{AD}$ is $\Pi_1^0$-hard for $d\geq 2$.
	\end{example}
	\begin{example}
		Let $G=\Z^{d}$, $d\geq2$. A $G$-SFT has topologically complete
		positive entropy (TCPE) when every topological factor is either a
		singleton with the trivial action by $G$, or has positive topological
		entropy. We claim that TCPE is a Berger property. Indeed, let $X_{-}$ be the SFT $\{0,1\}^{G}\cup\{2,3\}^{G}$.
		This system fails to have TCPE because it factors onto the SFT with
		exactly two configurations and zero topological entropy. The same is true for all extensions of $X_{-}$. Now let $X_{+}=\{0,1\}^G$. It is well known that this system has TCPE, and thus we have proved that  TCPE is a Berger property. This argument only shows that TCPE is $\Sigma_1^0$-hard. The complexity of this property is studied in detail in \cite{westrick_topological_2020}, where it is shown that it is $\Pi_1^1$-complete.
	\end{example}
	\begin{example}\label{ex:conjugacy}
		Let $G$ be amenable. The following are Berger properties for every nonempty SFT $X$: 
		\begin{itemize}
			\item The property  $\mathscr C (X)$ of being conjugate to $X$.
			\item The property $\mathscr F(X)$ of being a factor of $X$.
			\item  The property $\mathscr I(X)$ of embedding into $X$.
		\end{itemize}
		In the three cases it suffices to let $X_{+}=X$, and let $X_{-}$ be the disjoint union of $X$ with an SFT with nonzero topological entropy. Note that $\mathscr C(X)$ and $\mathscr I(X)$ are conjugacy-invariant counterparts of properties studied in \cite[Section 3]{cervelle_tilings_2004} (see also \cite{jeandel2015hardness}). Admitting an 
		embedding \textit{from} $X$ is not in general a Berger property, see \Cref{sec:remarks}. 
	\end{example} 
	It is natural to ask what is the  complexity of the properties considered here as  \Cref{thm:adian-rabin-for-sfts}  only shows that Berger properties are $\Dp(G)$-hard.
	\section{Uncomputability of dynamical invariants of SFTs}
	Recall that a real number $x$ is computable when there is an algorithm providing rational approximations to $x$ with any desired precision. We say that a real-valued dynamical invariant $\mathcal I$ for $G$-SFTs  is computable from presentations when there is an algorithm which given a presentation for $X$, provides rational approximations to $\mathcal I(X)$ with any desired precision.

	The fundamental dynamical invariant for $\Z^2$-SFTs is topological entropy. This invariant is in general not computable since there are $\Z^2$-SFTs whose topological entropy is a non-computable number \cite{hochman_characterization_2010}.  The same is true for the related invariant of entropy dimension \cite{meyerovitch_growthtype_2011,gangloff_characterizing_2022}. In view of these results, it is natural to ask whether some dynamical invariant of $\Z^2$-SFTs is computable.

	The main result of this section is a Rice-like theorem for all real-valued invariants satisfying mild monotony conditions:
	
	\begin{theorem}\label{invariants-monotone-by-products-and-unions-are-constants}
		Let $G$ be a finitely generated group with undecidable domino problem. Every computable dynamical invariant for $G$-SFTs that is nondecreasing by disjoint unions and products with nonempty systems is constant.   
	\end{theorem}
	\begin{proof}
		Let $\mathcal I$ be an invariant as in the statement, and suppose that  there are two SFTs $X_0$ and $Y_0$ with $\mathcal I(X_0)<\mathcal I(Y_0)$. Let $q$ be a rational number with $\mathcal I(X_0)<q<\mathcal I(Y_0)$. Given a possibly empty SFT $X$, we define $Z$ by 
		\[Z=X_0\sqcup Y_0\times X.\]
		By our assumptions on $\mathcal I$ we have $\mathcal I(Z)< q$ when $X=\emptyset$ and $\mathcal I(Z)> q$  when $X\ne\emptyset$. Since we can compute a presentation for $Z$ from a presentation for $X$ (\Cref{lem:products} and \Cref{lem:disjoint-unions}), and thanks to our assumption on the computability of $\mathcal I$, given $X$ we can determine in finite time whether $\mathcal I(Z)>q$ or $\mathcal I(Z)<q$. This amounts to determining whether $X=\emptyset$. This contradicts our hypothesis on $\Dp(G)$. Thus $\mathcal I(X_0)=\mathcal I (Y_0)$ and  $\mathcal I$ is constant. 
	\end{proof}
	It follows that for every amenable group with undecidable domino problem, topological entropy of SFTs is not computable from presentations. We mention that for some groups beyond $\mathbb{Z}^d$, $d\geq 2$ it is also known the existence of SFTs whose entropy is a non-computable real number  \cite{barbieri_entropies_2021,bartholdi_shifts_2024}. Our result has a much weaker conclusion, but its proof is remarkably simple and covers many other invariants. For instance, it covers invariants designed for zero-entropy systems that are similar to entropy dimension (see \cite{kanigowski_survey_2020}). Our result also holds if the invariant is only defined for nonempty subshifts or is only assumed to be computable for nonempty subshifts:
	\begin{remark}\label{computability-of-invariants-within-a-class}
		Let $\mathscr C$ be a class of $G$-SFTs such that (1) for every $X\in \mathscr C$ and nonempty SFT $Y$ we have $X\times Y\in\mathscr C$, and (2) $\mathscr C$ is closed by disjoint unions. Then  \Cref{invariants-monotone-by-products-and-unions-are-constants} holds within $\mathscr C$:  every dynamical invariant that is defined on $\mathscr C$, is computable in $\mathscr C$, and is  monotone by products and unions on $\mathscr C$, must have constant value on $\mathscr C$. The proof follows the same argument. 
		
		Some dynamical properties are known to imply the computability of topological entropy for $\Z^2$-SFTs  \cite{gangloff_effect_2019,pavlov_entropies_2015}. The observation in the previous paragraph shows that any such property can not verify (1) and (2).  
	\end{remark}
	We finish this section with a result for invariants taking values on abstract partially ordered sets. This result can be applied to recursion-theoretical invariants such as
	the Turing degree of the language of the SFT \cite{jeandel_turing_2013}.
	
	\begin{theorem}
		\label{thm:undecidability=000020of=000020monotone=000020invariants}Let $G$ be a finitely generated group with undecidable domino problem. Let $\mathcal{I}$ be a dynamical
		invariant for $G$-SFTs taking values in a partially ordered set $(\mathscr{R},\leq)$,
		which is non-increasing by factor maps, and whose value is minimal
		at the empty subshift. Then for every $r\in\mathscr{R}$ the following
		properties of a $G$-SFT $X$ are either trivial or undecidable: $\mathcal{I}(X)\geq r$, $\mathcal{I}(X)\leq r$, $\mathcal{I}(X)>r$, and $\mathcal{I}(X)<r$. 
	\end{theorem}
	
	\begin{proof}
		We consider in detail the property $\mathcal I(X)\leq r$, the other three cases are similar.  Let $\mathscr{P}$ be the property $\mathcal{I}(X)\leq r$ and observe that this property is preserved to factors. If the empty subshift verifies $\mathscr{P}$ then our claim follows  from \Cref{cor:rice-for-monotone-properties}. If the empty subshift does not verify $\mathscr{P}$ then this property is trivial: this follows from the transitivity of $\leq$, and the hypothesis that the value of $\mathcal{I}$ is minimal at the empty subshift.
	\end{proof}
	\section{Undecidability of dynamical properties of sofic subshifts}
	
	In this section we prove that if $G$ is a group with undecidable
	domino problem, then all nontrivial dynamical properties of sofic
	$G$-subshifts are undecidable from presentations. 
	
	We start by defining presentations for sofic subshifts. Let $S$ be a finite and symmetric generating set for $G$, and let $\pi\colon S^\ast\to G$ be defined by $\pi(w)=\underline w$. A \textbf{local function presentation} is a function $\mu\colon A^{W}\to B$,
	where $W$ is a finite subset of $\subset S^{\ast}$ and $A$ and $B$
	are finite subsets of $\N$. The local function $\mu_{0}$ presented
	by $\mu$ is defined as follows. We set $D=\pi(W)$, and define   $\mu_{0}\colon A^{D}\to B$
	by $\mu_{0}(p)=\mu(p\circ\pi)$. A \textbf{sofic $G$-subshift presentation} is a tuple $(A,\mathcal{F},\mu,B)$,
	where $A,B\subset\N$ are finite alphabets, $(A,\mathcal{F})$ is
	a $G$-SFT presentation, and $\mu\colon A^{W}\to B$, $W\subset S^{*}$
	is a local function presentation. The sofic subshift associated to
	this presentation $Y_{(A,\mathcal{F},\mu,B)}$ is the
	image of $X_{(A,\mathcal{F})}$ under the topological factor map with
	local function presented by $\mu$. 
	\begin{theorem}
		\label{thm:Rice-theorem-sofic-subshifts} Let $G$ be a finitely generated group with undecidable domino problem.  Then all nontrivial dynamical properties
		of sofic $G$-subshifts are undecidable.
	\end{theorem}
	
	\begin{proof}
		Let $\mathscr{P}$ be a nontrivial dynamical property of sofic subshifts. Replacing
		$\mathscr{P}$ by its negation if necessary, we can assume that the
		empty subshift does not satisfy the property. As $\mathscr{P}$ is
		nontrivial, there is a sofic subshift $Y_{+}$ satisfying $\mathscr{P}$. We
		fix this subshift for the rest of the proof. We also fix a presentation
		$(A_{+},\mathcal{F}_{+},\mu_{+},B_{+})$ for $Y_+$, and also $W_{+}\subset A^{\ast}$
		with $\mu_{+}\colon A_{+}^{W_{+}}\to B_{+}$. We define a computable function $f$ whose input is a presentation
		$(A,\mathcal{F})$ of a $G$-SFT, and whose output $f(A,\mathcal{F})$
		is the presentation of a sofic $G$-subshif such that $Y_{f(A,\mathcal{F})}$ has property
		$\mathscr{P}$ if and only if $X_{(A,\mathcal{F})}$ is nonempty. The existence of this function proves that $\mathscr{P}$ is undecidable, as otherwise it could be used to solve the domino problem for $G$.

		On input $(A,\mathcal{F})$, the function $f$ starts by computing a presentation $(A',\mathcal{F}')$,
		such that $X_{(A',\mathcal{F}')}$ is topologically conjugate to the
		direct product $X_{(A_{+},\mathcal{F}_{+})}\times X_{(A,\mathcal{F})}^ {}$. For this we use \Cref{lem:products}. Note that the alphabet $A'$ is equal to $\{\alpha(a_{+},a)\mid a_{+}\in A_{+},\ a\in A\}$, where $\alpha$ was defined in the proof of \Cref{lem:products}, and where we also fixed
		computable functions $\pi_{1},\pi_{2}$ satisfying $\pi_{i}(\alpha(n_{1},n_{2}))=n_{i}$,
		$i=1,2$. Next, we define a local function presentation $\mu\colon A'{}^{W_{+}}\to B_{+}$
		by $p\mapsto\mu(p)=\mu_{0}(\pi_{2}\circ p)$. Finally, the output
		of the computable function $f$ is $(A',\mathcal{F}',\mu',B_{+})$. It is clear that $f$ has the mentioned properties.
	\end{proof}
	In the language of many one reductions we proved that every nontrivial property of sofic subshifts $\mathscr{P}$ that is not verified by the empty subshift satisfies $\mathscr{P}\geq_m\Dp(G)$.
	\section{Further remarks}\label{sec:remarks}
	The results presented here can be used to show the undecidability of many dynamical properties of SFTs of common interest in the case of a group with undecidable domino problem. However, we consider that the frontier between decidability and undecidability is rather unclear as we know very little about the decidable region in the ``swamp of undecidability''. That is, we know very few decidable dynamical properties. 
	
	A natural generalization of \Cref{prop:decidable-fixed-points} is as follows. Given an SFT $X$, let $\mathscr E(X)$ be the property of admitting an embedding from $X$.  \Cref{prop:decidable-fixed-points} shows that this property is decidable when $|X|=1$, and the proof can be easily generalized to the case where $X$ is finite. The following question raises naturally:  
	\begin{question}
		Is there an infinite SFT $X$ such that $\mathscr E (X)$ is decidable?
	\end{question} 
	We also mention that in this work we have focused on properties preserved by conjugacy, but most of the arguments presented here can be adapted to set properties of SFTs. In this context an analogous of $\mathscr E(X)$ is the property $\mathcal C(X)$ of containing $X$. It turns out that the decidability of $\mathcal C(X)$ admits a characterization. The following argument was communicated to us by J.Kari. 
	\begin{proposition}
		$\mathcal C (X)$ is decidable if and only if the set 
		\[L(X)=\{p\colon W\to A \ \mid \text{$W\subset S^\ast$ is finite and $p$ appears in some $x\in X$}\}\]
		is decidable.
	\end{proposition}
	\begin{proof}
		Let $(A,\mathcal F)$ be a presentation for $X$, and suppose that $\mathcal C(X)$ is decidable. Then for every pattern presentation $p$ we have $p\not\in L(X)$ if and only if $X$ is contained in $X_{(A,\{p\})}$. This proves the forward implication. For the backward implication suppose that $L(X)$ is decidable, and let $(B,\mathcal G)$ be an SFT presentation.  It follows from the definitions that $(B,\mathcal G)$ verifies $\mathcal C(X)$ if and only if no element from $\mathcal G$ appears in $L(X)$. This is decidable by hypothesis and thus the backward implication is proved.    
	\end{proof}
	\section*{Acknowledgements}
	This work has benefited from helpful suggestions and remarks  of different people. I thank the anonymous reviewer, C.Rojas, S. Barbieri, J. Kari, B. Hellouin de Menibus, T. Meyerovich, and C. F. Nyberg-Brodda. This work was supported by  ANID 21201185, ANID/CENIA FB210017, MSCA 731143, and a grant from the Priority Research Area SciMat under the Strategic Programme Excellence Initiative at Jagiellonian University.
	\bibliographystyle{abbrv}
	\bibliography{bibliography}
\end{document}